\documentclass[10pt]{amsart}

\usepackage{amssymb,epsf,psfrag}
\usepackage{graphicx}
\usepackage{xcolor}
\usepackage{amsfonts}
\usepackage{enumerate}
\usepackage{latexsym}
\usepackage{epsfig}
\usepackage{amsthm}
\usepackage{amsmath}
\usepackage{latexsym}
\usepackage[all]{xy}
\usepackage{multirow}
\usepackage[
            breaklinks=true,
            colorlinks=true]{hyperref}
\usepackage[active]{srcltx}                                                                                                                                                                                                                                             

\usepackage{hyperref}

\makeatletter\@addtoreset{equation}{section}\makeatother
\makeatletter\@addtoreset{figure}{section}\makeatother
\makeatletter\@addtoreset{table}{section}\makeatother

\makeatletter\@addtoreset{equation}{section}\makeatother
\makeatletter\@addtoreset{figure}{section}\makeatother
\makeatletter\@addtoreset{table}{section}\makeatother

\newtheorem{theorem}{Theorem}[section]
\newtheorem{prop}[theorem]{Proposition}

\newtheorem{lemma}[theorem]{Lemma}

\newtheorem{theoremL}{Theorem}
\newtheorem{conjL}{Conjecture}

\newcommand{\R}{{\mathbb R}}
\newcommand{\Z}{{\mathbb Z}}

\newcommand{\op}[1]{\!\!\mathop{\rm ~#1}\nolimits}

\newenvironment{remark}{\refstepcounter{theorem}\par\medskip\noindent{\bf
Remark~\thetheorem~~}}{\unskip\nobreak\hfill\hbox{ $\oslash$}\par\bigskip}

\newenvironment{question}{\refstepcounter{theorem}\par\medskip\noindent{\bf
Question~\thetheorem~~}}{\unskip\nobreak\hfill\hbox{ $\oslash$}\par\bigskip}

\newenvironment{example}{\refstepcounter{theorem}\par\medskip\noindent{\bf
Example~\thetheorem~~}}{\unskip\nobreak\hfill\hbox{ $\oslash$}\par\bigskip}

\renewcommand{\geq}{\geqslant}
\renewcommand{\leq}{\leqslant}

\date{\today}

\begin{document}

\title[Fermat and the number of fixed points of periodic flows]{Fermat and the number of fixed points of periodic flows}

\author[L. Godinho]{Leonor Godinho}
\thanks{LG and SS were partially supported by Funda\c{c}\~ao para a Ci\^encia e Tecno\-logia (FCT / Portugal) through projects EXCL/MAT-GEO/0222/2012 and POCTI/MAT/117762/2010.}
\address{Departamento de Matem\'{a}tica, Centro de An\'{a}lise Matem\'{a}tica, Geometria e Sistemas Din\^{a}micos-LARSYS, Instituto Superior T\'ecnico, Av. Rovisco Pais 1049-001 Lisbon, Portugal}
\email{lgodin@math.ist.utl.pt}

\author[\'A. Pelayo]{ \'Alvaro Pelayo}
\thanks{ AP was partially
supported by an 
NSF CAREER DMS-1055897.}
 \address{ Department of Mathematics\\
 Washington University in St Louis\\
One Brookings Drive, Campus Box 1146\\
St Louis, MO 63130-4899, USA.}
\email{apelayo@math.wustl.edu}

\author[S. Sabatini]{Silvia Sabatini}
\thanks{SS was partially supported by the Funda\c{c}\~ao para a Ci\^encia e Tecno\-logia (FCT / Portugal) through the postdoctoral fellowship SFRH/BPD/86851/2012}
\address{Departamento de Matem\'{a}tica, Centro de An\'{a}lise Matem\'{a}tica, Geometria e Sistemas Din\^{a}micos-LARSYS, Instituto Superior T\'ecnico, Av. Rovisco Pais 1049-001 Lisbon, Portugal}
\email{sabatini@math.ist.utl.pt}

\begin{abstract} 
We obtain a general lower bound for the number of fixed points of a circle action on a compact almost complex manifold $M$
of dimension $2n$ with nonempty fixed point set,
provided the Chern number $c_1c_{n-1}[M]$ vanishes. The proof combines techniques originating in equivariant K\--theory with celebrated number theory results on polygonal numbers, 
introduced by Pierre de Fermat. 
This lower bound confirms in many cases a conjecture of Kosniowski from 1979, and is better than existing bounds for some symplectic actions.
Moreover, if the fixed point set is discrete, we prove divisibility properties for the number of fixed points,
improving similar statements obtained by Hirzebruch in 1999. 
Our results  apply, for example, to a class of manifolds which do not support any Hamiltonian circle action, namely those for which the first Chern class is torsion. 
This includes, for instance, all symplectic Calabi Yau manifolds. 
\end{abstract}

\maketitle

\section{Introduction}\label{sum}

Finding the minimal number of fixed points of a circle action on a compact manifold is one of the most pressing unsolved problems in equivariant geometry\footnote{In the terminology of dynamical systems, circle actions are regarded as
\emph{periodic flows} and the fixed points of the action correspond to the equilibrium points
of the flow.}.
It is deeply connected with the question of whether there exists a symplectic non-Hamiltonian $S^1$\--action
on a compact symplectic manifold with nonempty and discrete fixed point set, which is probably the hardest and most
interesting open problem
in the subject. Much of the activity concerning this question 
originated in a result by T. Frankel~\cite{Frankel1959} for K\"ahler manifolds, in which he showed that
a K\"ahler $S^1$\--action on a compact K\"ahler manifold $M$ is Hamiltonian if and only if it has fixed points. In this case,
this implies that the action has at least
$\frac{1}{2}\dim M +1$ fixed points, since they coincide with  
the critical points of the corresponding Hamiltonian function (a perfect Morse-Bott function).
For the larger class of unitary
manifolds, a conjecture in this direction was made  by Kosniowski \cite{Ko} in 1979 and is open in general.

\begin{conjL}[Kosniowski '79]\label{K1979}
Let $M$ be a $2n$-dimensional compact unitary $S^1$-manifold with isolated fixed points. 
If $M$ does not bound equivariantly then the number of fixed points is greater than $f(n)$, where $f(n)$ is some linear function.
\end{conjL}

In fact, Kosniowski suggested that, most likely, $f(n)$ is $\frac{n}{2}$ since this bound works in low dimensions, leading to the conjecture that the number of fixed points 
on $M$ is at least $\lfloor \frac{n}{2}\rfloor +1$.

Several other lower bounds  were obtained in the literature, by retrieving information from a nonvanishing  Chern number of the manifold. For example,
Hattori \cite{Ha84} showed that,  under a certain technical condition, if $c_1^n[M]$ does not vanish (implying that $c_1$ is not torsion)
then any $S^1$-action on an almost complex manifold preserving the a.c.\ structure
must have at least $n+1$ fixed points (see Theorem~\ref{thm:Hattori}). Since then many other results followed \cite{PeTo2010,LiLi2010,CKP12}; we review these in Section~\ref{s2}.
  
It is therefore natural to study the situation in which the first Chern class is indeed torsion.
In the symplectic case this condition automatically implies that the manifold cannot support any Hamiltonian circle action (see Proposition~\ref{prop:Ham}),
and is, for instance, satisfied by the important family of symplectic Calabi-Yau manifolds, for which we have $c_1=0$.
Since the existence of a symplectic manifold admitting a non-Hamiltonian circle action with discrete fixed point set is still unknown, and there is very little information 
on the required  topological properties of the possible candidates, our results shed some light on this problem.

In this note we make the weaker assumption that $c_1c_{n-1}[M]$ is zero (cf.\ Section \ref{the hypothesis}).
 The choice of this Chern number is motivated by its expression in terms
of numbers of fixed points obtained in \cite[Theorem 1.2]{GoSa12}.
Interestingly, on a compact symplectic manifold of dimension 6,
 a circle action is non-Hamiltonian if and only if $c_1c_2[M]= 0$ (cf. Proposition~\ref{cv}).  

Using  this expression, we show in Theorem~\ref{ab} that, whenever  $c_1c_{n-1}[M]$ vanishes,
a \emph{lower bound} $\mathcal{B}(n)$ for the number of fixed points of a circle action on an almost complex manifold can be obtained from the minimum values of certain 
integer-valued functions restricted to a set of integer points in  a specific hyperplane. 
These minimization problems are then solved in Theorems~\ref{thm:B} and \ref{thm:C}, using celebrated number theory results  on the possible representations of a  positive integer number as a sum of  polygonal numbers (namely  squares and triangular numbers). These were originally stated by Fermat in 1640 and proved by Legendre, Lagrange, Euler, Gauss and  Ewell (see Section \ref{sec:prent}).
The lower bounds obtained  are summarized in the following theorem.

\begin{theoremL}\label{thm:MAIN} 
Let $(M,J)$ be a $2n$-dimensional compact connected almost complex manifold equipped with a $J$-preserving  $S^1$\--action with 
nonempty fixed point set and 
such that $c_1c_{n-1}[M]=0$.
Then the number of fixed points of the $S^1$-action is at least $\mathcal{B}(n)$, where $\mathcal{B}(n)$ is given as follows. 

For \framebox{$n=2m$} and $r:=\gcd(m,12)$, 
\begin{enumerate}
\itemsep2pt \parskip0pt \parsep0pt

\item[{\rm (i)}] if $r=1$ then $\mathcal{B}(n)=12$;
\item[{\rm (ii)}] if $r=2$ then

\begin{tabular}{ll}
$  \bullet \,  \mathcal{B}(n)=6$  & if $n \not\equiv 28 \pmod{32}$, \\
 $\bullet \,  \mathcal{B}(n)=12$ & otherwise;
 \end{tabular}

\item[{\rm (iii)}] if $r=3$ then 

\begin{tabular}{ll}
$\bullet \,  \mathcal{B}(n)= 4$ & if all  prime factors of $\frac{n}{6}$ congruent to $3\pmod 4$ \\ & occur  with even exponent, \\
$\bullet \,  \mathcal{B}(n)= 8$ & otherwise;
 \end{tabular}

\item[{\rm (iv)}] if $r=4$ then 

\begin{tabular}{ll}
 $\bullet \, \mathcal{B}(n)= 3$ & if $\frac{n}{2}$ is a square, \\
 $\bullet \,  \mathcal{B}(n)= 6$  & if $\frac{n}{2}$ is not a square and $n\neq 4^k(8t+7)\, \forall k,t\in  \mathbb{Z}_{\geq0},$ \\
$\bullet \,  \mathcal{B}(n) = 9$ & otherwise;
 \end{tabular}

\item[{\rm (v)}] if $r=6$ then 

\begin{tabular}{ll}
$\bullet \,  \mathcal{B}(n)= 2$ & if $\frac{n}{12}$ is a square, \\
$\bullet \,  \mathcal{B}(n)=4$ & if $\frac{n}{12}$ is not a square and all prime factors of $\frac{n}{6}$ \\ & congruent to   $3\pmod 4$ occur with even exponent, \\
$\bullet \, \mathcal{B}(n)= 6$ & if $\frac{n}{12}$ is not a square,  at least one prime factor of $\frac{n}{6}$ \\ & congruent  to $3\pmod 4$ occurs with an odd exponent  \\  & and $n \not\equiv 28 \pmod{32}$, \\
$\bullet \,  \mathcal{B}(n)= 8$ & otherwise; 
 \end{tabular} 

\item[{\rm (vi)}] if $r=12$ then 

\begin{tabular}{ll}
$\bullet \, \mathcal{B}(n)= 2$ & if $\frac{n}{12}$  is a square, \\
$\bullet \, \mathcal{B}(n)=3$ &  if $\frac{n}{2}$ is a square, \\
$\bullet \, \mathcal{B}(n)= 4$  & if none of the above holds and all prime factors  of $\frac{n}{6}$ \\ & congruent to $3\pmod 4$ occur with even exponent, \\
$\bullet \,\mathcal{B}(n)=6$  & if none of the above holds  and $n\neq  4^k(8t+7)\, \forall k,t\!\in \!\mathbb{Z}_{\geq0}$,\\
$\bullet \, \mathcal{B}(n)= 7$ & otherwise.
\end{tabular}

\end{enumerate}
For \framebox{$n=2m+1$} and $r:= \gcd(m-1,12)$,

\begin{enumerate}
\itemsep2pt \parskip0pt \parsep0pt
\item[{\rm (i)}] if $r\leq 4$ then $\mathcal{B}(n)=\frac{24}{r}$;
\item[{\rm (ii)}] if $r=6$ then 

\begin{tabular}{ll}
$  \bullet \, \mathcal{B}(n)= 4$ & if  every prime factor of $\frac{n}{3}$ congruent to $3\pmod 4$ \\ & occurs with even exponent,  \\
 $\bullet \, \mathcal{B}(n)= 8$ & otherwise; 
\end{tabular}

\item[{\rm (iii)}] if $r=12$ then 

\begin{tabular}{ll}
$  \bullet \,\mathcal{B}(n)= 2$ & if $\frac{n-3}{24}$  is a triangular number, \\
$  \bullet \, \mathcal{B}(n)=4$ & if  $\frac{n-3}{24}$ is not a triangular number and every prime \\ & factor of $\frac{n}{3}$   congruent to $3\pmod 4$ occurs with  \\ &  even exponent,  \\
$  \bullet \, \mathcal{B}(n)= 6$  & otherwise.
\end{tabular}
\end{enumerate}
\end{theoremL}

In some dimensions, the lower bounds obtained confirm Kosniowski's  conjecture, and, in some cases, they are better than $n+1$,  the existing lower bound  for Hamiltonian and some 
almost complex actions. We give a complete list of these dimensions in Propositions~\ref{prop:2} 
and \ref{prop:3}.

Under the same vanishing condition on $c_1c_{n-1}[M]$, the expression of this Chern number given in \cite[Theorem 1.2]{GoSa12} can  also be used  to prove 
\emph{divisibility results} for the number of fixed points, provided that the fixed point set is discrete. 
\begin{theoremL}\label{thm:D}
Let $(M,J)$ be a $2n$-dimensional compact connected almost complex manifold equipped with a $J$-preserving  $S^1$\--action with 
nonempty, discrete fixed point set $M^{S^1}$ and 
such that $c_1c_{n-1}[M]=0$. Then if  $n=2m$  is even,
\begin{equation}\label{eq:div1}
\lvert M^{S^1} \rvert\equiv  0 \pmod{ \frac{12}{r}} \quad \text{with $r=\gcd{(m,12)}$}
\end{equation}
and, if $n=2m+1$ is odd, 
\begin{equation}\label{eq:div2}
\lvert M^{S^1} \rvert\equiv 0 \pmod{ \frac{24}{r}} \quad \text{with $r=\gcd{(m-1,12)}$.}
\end{equation}
\end{theoremL}

\begin{remark}
In particular, if $n=\frac{1}{2}\dim(M)$ is odd then $\lvert M^{S^1} \rvert$ is always even.
\end{remark}

Note that, when the fixed point set is discrete, the number of fixed points coincides with the Euler characteristic of the a.c. manifold. 
Coincidently,   in  a letter to V.\ Gritsenko,  Hirzebruch  also obtains divisibility results for the Euler characteristic of an almost complex manifold $M$ satisfying  $c_1c_{n-1}[M]=0$ \cite{Hi99}.  
By using our methods we are able to improve his results whenever $\dim M\not\equiv 0 \pmod{6}$ (cf. Theorem~\ref{thm:E}). 
Under the stronger condition that $c_1=0$ in integer cohomology, we can combine Hirzebruch's results with ours obtaining, in some cases, 
a better lower bound for the number of fixed  points (see Theorems~\ref{thm:F} and \ref{thm:G}). 
For example, when $\dim M=4$ and $c_1=0$ we prove that the number of fixed points is at least $24$. 
This will be true, in general, whenever $\dim M \equiv 4 \pmod{16}$ and $\dim M \not \equiv 0 \pmod 6$.


In Section~\ref{sec:ex} we provide several examples that show how some of the lower bounds obtained are sharp and illustrate our divisibility results for the number of fixed points. 
In particular, we give examples of sharp lower bounds in dimensions $4, 6, 10, 12$ and $18$. It would be interesting to know the answer to the following question.
\begin{question}
\emph{Does there exist a compact almost complex $S^1$-manifold $M$ of dimension $8$ with $c_1c_3[M]=0$ and with exactly $6$ fixed points?}
\end{question}
The existence of such a manifold would also provide an example with a sharp lower bound in dimension $14$. 
In the following table we illustrate some of the results obtained in this work.
\vspace{0.3cm}

\renewcommand{\arraystretch}{1.1}
\begin{tabular}{| l ||  l | c | c  |} 
\hline  & & & \\ \multirow{3}{*}{ {\color{blue!60!black} $\dim M= 2n $}}  & Possible values   & Kosniowski's  & Lower bound for  \\
 & of $\lvert M^{S^1}\rvert$ if  & lower bound  & Hamiltonian actions \\ 
 & $c_1c_{n-1}[M]=0$ &  $\lfloor \frac{n}{2} \rfloor + 1$ &  $n+1$\\ & & & \\ \hline 
\hline {\color{blue!60!black} $4^*$} & {\bf 12}, 24, 36, \ldots  & 2 & 3 \\  \hline 
{\color{blue!60!black} $6$}  & {\bf 2}, 4, 6, \ldots & 2  & 4 \\  \hline 

{\color{blue!60!black} $8$}  &  {\bf 6}, 12, 18, \ldots & 3 & 5 \\  \hline 

{\color{blue!60!black} $10$} &  {\bf 24}, 48, 72, \ldots & 3 & 6  \\  \hline 

{\color{blue!60!black} $12$} &  {\bf 4}, 8, 12, \ldots & 4 & 7 \\  \hline 

{\color{blue!60!black} $14$} &   {\bf 12}, 24, 36, \ldots  & 4 & 8 \\  \hline 

{\color{blue!60!black} $16$} &   {\bf 3}, 6, 9, \ldots & 5 & 9 \\  \hline 

{\color{blue!60!black} $18$}  &  {\bf 8}, 16, 24, \ldots & 5 & 10\\  \hline 

{\color{blue!60!black} $20^*$}&  {\bf 12}, 24, 36, \ldots &  6 & 11\\  \hline 

{\color{blue!60!black} $22$} &  {\bf 6}, 12, 18, \ldots & 6 & 12\\   \hline 

{\color{blue!60!black} $24$} &  {\bf 2}, 4, 6, \ldots & 7 & 13 \\  \hline 

{\color{blue!60!black} $26$} &  {\bf 24}, 48, 72, \ldots  & 7 & 14\\  \hline 

{\color{blue!60!black} $28$}  &  {\bf 12}, 24, 36, \ldots & 8 & 15\\  \hline 

{\color{blue!60!black} $30$} &  {\bf 4}, 8, 12, \ldots & 8 & 16\\   \hline 

\multicolumn{4}{l}{* if $c_1=0$ then  the possible values of $\lvert M^{S^1}\rvert$ are  {\bf 24}, 48, 72, \ldots }
 
\end{tabular}

\renewcommand{\arraystretch}{1}

{\emph{Acknowledgements}.  
This paper started at the Bernoulli Center
in Lausanne (EFPL) during the program on Semiclassical Analysis and Integrable
Systems organized by \'Alvaro Pelayo, Nicolai Reshetikhin, and   San V\~u Ng\d oc, from
July 1 to December 31, 2013.    We would like to thank D.\ McDuff and T.\ S.\ Ratiu for useful comments and discussions}.

\section{Preliminares} \label{s2}

We review some results which are relevant  for this article, including some which we will need in the proofs.

\subsection{Origins}

It has been a long standing problem to estimate the minimal number of fixed points of a smooth  circle action on a compact smooth manifold with nonempty fixed point set. If the manifold is symplectic, i.e.\ if it admits 
 a closed, non\--degenerate
two\--form $\omega \in \Omega^2(M)$ (\emph{symplectic form}), we say that an $S^1$\--action on  $(M,\omega)$ 
is \emph{symplectic} if it preserves $\omega$.  If
$\mathcal{X}_M$ is the vector field induced by the $S^1$ action then we say that the action is \emph{Hamiltonian} if the $1$\--form
$\iota_{\mathcal{X}_M} \omega:=\omega(\mathcal{X}_M,\cdot)$ is exact, that is,  if there exists 
 a smooth map $\mu \colon M \to \R$ such that
$- \op{d}\! \mu = \iota_{\mathcal{X}_M} \omega.$
The map $\mu$ is called a \emph{momentum map}.  
If a symplectic manifold is equipped with a Hamiltonian $S^1$-action then the following fact is well-known (cf. Section~\ref{fr}).

\begin{prop}\label{prop}
A Hamiltonian $S^1$-action on a $2n$-dimensional compact symplectic manifold has at least $n+1$ fixed points.
\end{prop}

Since Frankel \cite{Frankel1959} showed that 
a K\"ahler $S^1$-action on a compact K\"ahler manifold is Hamiltonian
if and only if it has fixed points,
a lower bound for the number of fixed points of a circle action was thus obtained in the K\"ahler case.

\begin{theorem}[Frankel '59]\label{general2}
A K\"ahler $S^1$-action on a $2n$-dimensional compact connected K\"ahler manifold is Hamiltonian if and only if it has fixed points, in which case it has at least $n+1$ fixed points.
\end{theorem}

Trying to extend this lower bound on the number of fixed points, Kosniowski considers unitary $S^1$-manifolds and proposes that if $M$ is a $2n$-dimensional unitary $S^1$-manifold with isolated fixed points which does not bound equivariantly, then the number of fixed points is greater than $f(n)$, where $f(n)$ is some linear function (Conjecture~\ref{K1979}).
In fact, Kosniowski suggests that, most likely, one has $f(n)=\frac{n}{2}$ since it works in low dimensions, leading to the conjecture that the number of fixed points on $M$ is at least $\lfloor \frac{n}{2}\rfloor +1$, where
$2n$ is the dimension of $M$.
\begin{remark}
Recall that a \emph{unitary $S^1$-manifold} is a stable complex manifold $M$, i.e.\ $TM\oplus \R^{2k}$ has a complex structure for some $k$ (where $\R^{2k}$ is the trivial real vector bundle over $M$), equipped with an $S^1$-action that preserves this structure.  Thus,  every $S^1$-almost complex manifold, and hence every $S^1$-symplectic manifold, is unitary. Moreover,  $M$ \emph{bounds} if it is cobordant with the empty set, meaning that it can be realized as the oriented boundary of a smooth oriented ${2n+1}$-manifold with boundary. In particular, if this is the case, all the Pontrjagin and Stiefel-Whitney numbers vanish. Note that  if $M$ is smooth and admits a semi-free circle action with isolated fixed points then it bounds \cite{LiLi2010}. 
\end{remark}

Another related result was proved by Hattori \cite[Theorem 5.1]{Ha84} for almost complex manifolds.  If $(M,J)$ is a $2n$-dimensional \emph{almost-complex manifold}, i.e.\ $J$ is a complex structure on the tangent bundle $TM$,  then one can consider the Chern classes $c_j\in H^{2j}(M,\Z)$ of $TM$  as well as any Chern number.  In particular, one can take  $c_1^n[M]$. On the other hand,  by taking   the restriction of the first equivariant Chern class of $TM$
 at  each fixed point $p \in M^{S^1}$, which one can naturally identify with the sum of the weights of the $S^1$-isotropy representation  $T_{p}M$,  one obtains a map
 $$
 c_1^{S^1}(M) \colon M^{S^1} \to \Z, \,\,\,\,\,\,\,p \mapsto c_1^{S^1}(M)(p) \in \Z,
 $$ 
 called  the \emph{Chern class map} of $M$.

\begin{theorem}[Hattori]\label{thm:Hattori}
If $M$ is a $2n$-dimensional almost-complex manifold such that $c_1^n[M]$  does not vanish and the Chern class map is injective,  then any $S^1$-action on $M$ must have at least $n+1$ fixed points.
\end{theorem}

\begin{remark}
This result is improved in \cite[Corollary 1.5]{LiLi2010} where the condition on the injectivity of the Chern class map is removed.
\end{remark}

\subsection{Some recent contributions}

Following Kosniowski's conjecture and the theorems of Frankel and Hattori, many results have appeared in recent works. 

By using the Atiyah\--Bott and Berline\--Vergne localization formula in equivariant
cohomology,
Pelayo and Tolman \cite{PeTo2010} proved the following result which generalizes the known lower bound for the number of fixed points of Hamiltonian $S^1$-actions to some symplectic   non-Hamiltonian actions. Note, however, that there are no known examples of these actions with discrete fixed point sets.
\begin{theorem}[\cite{PeTo2010}]\label{general}
Let $S^1$  act symplectically
on a compact  symplectic $2n$-dimensional manifold $M$.
If the Chern class map $c_1^{S^1}(M) \colon M^{S^1}\to \Z$ is somewhere injective\footnote{Let $f\colon X \to Y$ be 
a map between sets then
$f$ is \emph{somewhere injective}
if 
there exists  $y\in Y$ such that $f^{-1}(\{y\})$
is the singleton.}  then
 the $S^1$\--action has at least  $n + 1$ fixed points.
\end{theorem}

In particular, they obtain the following lower bounds for the number of fixed points.

\begin{theorem}[\cite{PeTo2010}]\label{maintheorem}
Let $S^1$ act symplectically
on a compact  symplectic manifold $M$ with nonempty fixed point set. Then
there are at least  two fixed points. In particular,
\begin{itemize}
\item if
$\dim M\ge 8$, 
then  there exist at least three fixed points;
\item if $\dim M\ge 6$, and $c_1^{S^1}(M) \colon M^{S^1}\to \Z$ is not identically zero,
 then there exist at least four fixed points.
 \end{itemize}
\end{theorem}

Following this result,
Ping Li and Kefeng Liu generalized Hattori's theorem  \cite{LiLi2010}.
\begin{theorem}[Li\--Liu]\label{thm:PL}
Let $M^{2mn}$ be an almost-complex manifold. If there exist positive integers $\lambda_1,\ldots,\lambda_u$ with  $\sum_{i=1}^u \lambda_i= m$ such that the corresponding Chern number $(c_{\lambda_1}\cdots c_{\lambda_u})^n[M]$ is nonzero, then any $S^1$-action on $M$ must have at least $n+1$ fixed points.
\end{theorem}

Another related result was  obtained by Cho, Kim and Park \cite{CKP12}.
\begin{theorem}[Cho-Kim-Park]
Let $M$ be a $2n$-dimensional unitary $S^1$-manifold and let $i_1,i_2,\ldots, i_n$  be non-negative integers satisfying $i_1+2i_2+\cdots+ni_n=n$. If $M$ does not bound equivariantly and $c_1^{i_1}c_2^{i_2}\cdots c_n^{i_n}\neq 0$ then $M$ must have at least $\text{max}\{i_1,\cdots,i_n\}+1$ fixed points. 
\end{theorem}



\section{Fermat's famous statements}
\label{sec:prent}

In 1640 Fermat stated (without proof) that  every positive integer is a sum of at most $4$ squares and a sum of at most $3$ triangular numbers, where square and triangular numbers are those respectively described by $k^2$ and $\frac{k(k+1)}{2}$, with $k=0,1,2,3,\ldots$.  

Lagrange, in 1770,  proved the part of Fermat's theorem regarding squares, obtaining his celebrated Four Squares Theorem \cite[p. 279]{D}.

\begin{theorem}[Lagrange's Four Squares Theorem]\label{thm:Lagrange}
Every nonegative integer is the sum of $4$ or fewer squares.
\end{theorem}

In 1798 Legendre proved a much deeper statement which described exactly which numbers needed all four squares  \cite[p. 261]{D}.
\begin{theorem}[Legendre's Three Squares Theorem]\label{thm:Legendre}
The set of positive integers that are not sums of three or fewer squares is the set
$$
\left\{m\in \Z_{>0}:\,\, m=4^k(8t+7), \quad \text{for some}\quad k,t\in \Z_{\geq0} \right\}.
$$
\end{theorem}

After this, it was natural to think which numbers could be written as sum of two squares. A complete answer to this question was given by Euler  \cite[p. 230]{D}.
\begin{theorem}[Euler]\label{thm:Euler}
A positive integer $m>1$ can be written as a sum of two squares if and only if every prime factor of $m$ which is congruent to $3 \pmod 4$ occurs with even exponent. 
\end{theorem}

\begin{example}
The integer  $m={\bf 245}=5\cdot 7^2$ can be written as a sum of two squares. In particular,  
$245= 4\cdot7^2 + 7^2 = 14^2 + 7^2.$
As the number $m={\bf 105}$ is not divisible by $4$ and is congruent to $1 \pmod 8$, one concludes that it is the sum of $3$ or fewer squares. However,  since $105=3\cdot 5\cdot 7$ has a prime factor congruent to $3 \pmod 4$ occurring with odd exponent, it cannot be written as a sum of $2$ squares. For instance, we have
$
105=10^2+2^2+1^2.
$
Since $m={\bf 60}=4\cdot 15 = 4\cdot (8+7)$, we know from Theorem~\ref{thm:Legendre} that it cannot be represented as a sum of $3$ or fewer squares so we really need $4$ squares. For example,
$
60=6^2+4^2+2^2+2^2.$
\end{example}

Let us now see what happens with triangular numbers. The part of Fermat's statement regarding these numbers was first proved by Gauss \cite[p. 17]{D}.
\begin{theorem}[Gauss]\label{thm:Gauss}
Every nonegative number is the sum of three or fewer triangular numbers. 
\end{theorem}

After this result, Ewell \cite{E} gave a simple description of those numbers that are sums of two triangular numbers.
\begin{theorem}[Ewell]\label{thm:Ewell}
A positive integer $m$ can be represented as a sum of two triangular numbers if and only if every prime factor of $4m+1$ which is congruent to $3 \pmod 4$ occurs with even exponent. 
\end{theorem}

\begin{example}
Taking $m={\bf 106}$ one obtains $4m+1= 425 = 5^2\cdot 17$ and so $m$ can be written as a sum of two triangular numbers. For instance,
$
106= 105 +1= \frac{14 \cdot 15}{2}+\frac{1\cdot 2}{2}.
$
On the other hand, if one takes $m={\bf 59}$, then $4m+1= 237=3\cdot 79$ and so, by Theorem~\ref{thm:Ewell}, $m$  cannot be written as a sum of $2$ triangular numbers. For instance we have
$
59=28+21+10=\frac{7\cdot 8}{2}+\frac{6\cdot 7}{2}+\frac{4\cdot 5}{2}.
$
\end{example}
\section{A minimization problem} \label{s3}
\subsection{The hypothesis $c_1c_{n-1}[M]=0$}\label{the hypothesis}
Let us now return to our initial problem of obtaining lower bounds for the number of fixed points of an almost complex circle action.
In most of  the results presented in Section~\ref{s2} the  lower bounds for the number of fixed points of a  circle action are obtained by retrieving information from  a nonvanishing Chern  number. In Theorem~\ref{general}, the crucial hypothesis for the establishment of the lower bound is the existence of  a value  $k$ of the Chern map $c_1(M)$ for which
$$
\sum_{\tiny{\begin{array}{c}p\in M^{S^1}\\ c_1(M)(p)=k\end{array}}} \frac{1}\Lambda_p \neq 0,
$$ 
where $\Lambda_p$ is the product of the weights in the isotropy representation $T_pM$. This is trivially achieved whenever the Chern map is somewhere injective \cite{PeTo2010} but also when $c_1^n[M]\neq 0$ as pointed out in \cite[Lemma 3.1]{LiLi2010}, leading to Theorem~\ref{thm:Hattori}.

In this work we focus on the situation in which a particular Chern number vanishes.    The only known expressions of Chern numbers in terms of  number of fixed points concern $c_n[M]$ (which equals the Euler characteristic and  the number of fixed points, when the fixed point set is  discrete) and  $c_1c_{n-1}[M]$ \cite[Theorem 1.2]{GoSa12}. Thus the natural candidate is $c_1c_{n-1}[M]=0$.

Note  that $c_1c_{n-1}[M]=0$ is  satisfied under the stronger condition that  $c_1$ or $c_{n-1}$ are torsion in integer cohomology. 
In the case in which $c_1$ is torsion, the following Lemma proves that the results in \cite{Ha84, PeTo2010} cannot be applied.
\begin{lemma}\label{lemma:torsion}
Let $(M,J)$ be a compact almost complex manifold such that  $c_1$ is a torsion element in $H^2(M,\mathbb{Z})$. If $M$ admits a $J$-preserving circle action with a discrete fixed point set then the Chern class map $c_1^{S^1}(M):M^{S^1}\to \mathbb{Z}$ is identically zero.
\end{lemma}
\begin{proof}
Since  $c_1$ is a torsion element in $H^2(M,\mathbb{Z})$, there exists $k\in \mathbb{Z}$ such that $kc_1=0$. Then the restriction of the equivariant extention $k\, c_1^{S^1}\in H_{S^1}^2(M,\mathbb{Z})$  to the fixed point set is constant, implying that the Chern class map is constant. Since $c_1^{S^1}(M)(p)$ coincides with the sum of the isotropy weights at $p\in M^{S^1}$, Proposition 2.11 in \cite{Ha84} implies that  
$$\sum_{p\in M^{S^1}} c_1^{S^1}(M)(p)=0$$  
and so  this constant must be zero.
\end{proof}

Note that if $M$ is a $6$-dimensional compact connected symplectic manifold, the action is Hamiltonian if and only if $c_1c_2[M]\neq 0$. Indeed, we have the following proposition.

\begin{prop} \label{cv}
Suppose that $S^1$ acts symplectically 
on a compact  connected $6$\--dimensional symplectic manifold
$M$ with nonempty discrete fixed point set. Then the $S^1$\--action is Hamiltonian
if and only if $c_1c_2[M]\neq 0$.
\end{prop}

\begin{proof} 
The result follows from \cite{Feldman2001} and the fact that,  when $\dim(M)=6$, one has
$${\rm Todd}(M)=\displaystyle\int_{M}\frac{c_1c_2}{24}.$$ 
\end{proof}

In general, if the manifold is symplectic and  $c_1$ is torsion in integer cohomology, then, necessarily, the action is non-Hamiltonian. 
\begin{prop}\label{prop:Ham}
Let $(M,\omega)$ be a compact symplectic manifold such that $c_1$ is torsion in integer cohomology. Then $M$ does not admit any Hamiltonian circle action with isolated fixed points.
\end{prop}
\begin{proof}
This follows immediately from Proposition 3.21 in  \cite{Ha84}, by using  a result of Feldman \cite{Feldman2001}  which states that the Todd genus associated to $M$ is either $1$  or $0$, according to whether 
the action is Hamiltonian or  not. Alternatively we can use Lemma~\ref{lemma:torsion} since, if the action is Hamiltonian, then $c_1^{S^1}(M)(p)\neq0$  at both the minimum and the maximum points of the momentum map.  
\end{proof}

Therefore, our results naturally apply to a class of compact symplectic manifolds that do not support any Hamiltonian circle action with isolated fixed points, namely symplectic manifolds whose first Chern class is torsion. 
 For example, symplectic Calabi Yau manifolds, i.e.\ symplectic manifolds with $c_1=0$ \cite{FP09}.

\subsection{Tools}
Let us then see how to obtain a lower bound for the number of fixed points of a $J$-preserving circle action on an almost complex manifold $(M,J)$ satisfying $c_1c_{n-1}[M]=0$. 

The first result that we need is the expression of $c_1c_{n-1}$ in terms of  numbers of fixed points.
\begin{theorem}[\cite{GoSa12}]\label{GS}
Let $(M,J)$ be a $2n$-dimensional compact connected almost complex  manifold equipped with an $S^1$\--action which preserves the almost complex structure $J$ and has a nonempty discrete fixed point set. For every $i = 0,\ldots, n$, let $N_i$ be the number of fixed points with exactly $i$ negative weights in the isotropy representation $T_pM$. Then
\begin{equation}\label{eq:GS}
c_1 c_{n-1} [M]:= \int_M c_1  c_{n-1}=\sum_{i=0}^n N_i\Big(6i(i-1)+\frac{5n-3n^2}{2}  \Big),
\end{equation}
where $c_1$ and $c_{n-1}$ are respectively the first and the $(n-1)$ Chern classes of $M$.
\end{theorem}

\begin{remark}
If $M$ is a $2n$-dimensional symplectic and the $S^1$-action is Hamiltonian then the number $N_i$ of fixed points with exactly $i$ negative weights in the corresponding isotropy representations coincides with the $2i$-th Betti  number $b_i(M)$ of $M$.  Consequently the expression for $c_1 c_{n-1} [M]$ given in \eqref{eq:GS} becomes
\begin{equation}\label{eq:GS2}
c_1 c_{n-1} [M] =\sum_{i=0}^n b_{2i}(M)\Big(6i(i-1)+\frac{5n-3n^2}{2}  \Big).
\end{equation}
For example, if $\dim M=4$, equation \eqref{eq:GS2} gives
\begin{equation}\label{eq:GS3}
c_1^2 [M] = 10b_0(M)-b_2(M),
\end{equation}
where we used the fact that $b_0(M)=b_4(M)$.
\end{remark}

\subsection{The minimization problem}

For each  $m \in \Z_{\geq 0}$ let us consider the functions $F_1,F_2:\Z^{m+1}\to \Z$ defined by
\begin{eqnarray}
F_1(N_0,\ldots,N_m)&:=& N_m+ 2\, \sum_{k=1}^{m} N_{m-k}; \label{f1} \\
F_2(N_0,\ldots,N_m)&:=&  2\, \sum_{k=0}^{m} N_{k}; \label{f2} \\
G_1(N_0,\ldots,N_m)&:=& -mN_m+ 2\, \sum_{k=1}^{m} (6k^2-m) \, N_{m-k} ;\nonumber\\
G_2(N_0,\ldots,N_m)& := & \sum_{k=0}^{m}\Big(6k(k+1)-(m-1)\Big)N_{m-k}.\nonumber
\end{eqnarray} 
Moreover, for $i \in \{1,\,2\}$, let
\begin{equation} \label{z}
\mathcal{Z}_i:=\left\{(N_0,\ldots,N_m) \in (\Z_{\geq 0})^{m+1}\mid  F_i(N_0,\ldots,N_m)>0,\,\, G_i(N_0,\ldots,N_m)=0  \right\}. 
\end{equation}
Then we have the following result.

\begin{theoremL} \label{ab}
Let $(M,J)$ be a $2n$-dimensional compact connected almost complex manifold equipped with a $J$-preserving  $S^1$\--action with 
nonempty, discrete fixed point set and 
such that $c_1c_{n-1}[M]=0$. 

For $m:=\lfloor\frac{n}{2}\rfloor$, let $F_1,\,F_2:\Z^{m+1}\to \Z$ be the functions defined 
respectively  in \eqref{f1} and  \eqref{f2}, and let $\mathcal{Z}_1,\mathcal{Z}_2$ be the sets given in \eqref{z}.
Then the  $S^1$\--action has at least $\mathcal{B}(n)$ fixed points, where
$$
\mathcal{B}(n):= 
\left\{ \begin{array}{rl}
   \min_{\mathcal{Z}_1}
   F_1 & \,\,\,\,\,\,\,\,\,\,\,\,\,\,\, \text{if $n$ is even};   \\
    \textup{\,} \\
    \min_{\mathcal{Z}_2}
 F_2  & \,\,\,\,\,\,\,\,\,\,\,\,\,\,\, \text{if $n$ is odd}.
  \end{array}   \right.
$$
\end{theoremL}

\begin{proof}
By Proposition 2.11 in \cite{Ha84}, we know that $N_i=N_{n-i}$ for every $i \in \Z$.  Thus, as the
 total number of fixed points is $$\sum_{k=0}^{n}N_k,$$ 
it follows that 
$F_1(N_0,\ldots,N_m)$ and $F_2(N_0,\ldots,N_m)$ count the total number of fixed points 
when $n=2m$ and $n=2m+1$ respectively, and, since the fixed point set is nonempty, 
we must have $F_1>0$ and $F_2>0$.

Since we are assuming that $c_1c_{n-1}[M]=0$, the constraints $G_1=0$ and $G_2=0$
are obtained from  Theorem~\ref{GS}, according to whether $n$ is odd or even. Indeed, let $g:\mathbb{Z}\times \mathbb{Z} \to \mathbb{Z}$ be the map
$$
g(i,n)=6i(i-1)+\displaystyle\frac{5n-3n^2}{2}.
$$ 
If $n=2m$, since $N_i=N_{n-i}$, we have
\begin{equation}\label{G1=0}
\begin{aligned}
0= \sum_{i=0}^nN_i\, g(i,n) & =   -mN_m+\sum_{k=1}^{m}  \Big(g(m-k,2m)+g(m+k,2m)\Big) \, N_{m-k} \\
 & =-mN_m+2\sum_{k=1}^{m} (6k^2-m) \, N_{m-k} \\
 &=G_1(N_0,\ldots,N_m). \nonumber
\end{aligned}
\end{equation}
Analogously, if $n=2m+1$, we have
\begin{equation}\label{G2=0}
\begin{aligned}
0= \sum_{i=0}^nN_i\, g(i,n) & =  \sum_{k=0}^{m}N_{m-k} \Big( g(m-k,2m+1)+g(m+k+1,2m+1) \Big) \\
& =2\sum_{k=0}^{m}\Big(6k(k+1)-m+1\Big) \,N_{m-k} \\
&=2\, G_2(N_0,\ldots,N_m). \nonumber
\end{aligned}
\end{equation}
\end{proof}

\section{A lower bound when $n$ is even}
Here we compute the minimal value $\mathcal{B}(n)$ of the function $F_1$ restricted to $\mathcal{Z}_1$,  obtaining a lower bound for the number of fixed points of the $S^1$-action when $n$ is \emph{even}.

\begin{theoremL}\label{thm:B} Let $n=2m$ be an even positive integer and let $\mathcal{B}(n)$ be the minimum of the function $F_1$ restricted to $\mathcal{Z}_1$, where $F_1$ and $\mathcal{Z}_1$ are respectively defined by \eqref{f1} and \eqref{z}. Then $\mathcal{B}(n)$ can take all values in the set $\{2,3,4,6,7,8,9,12\}$. In particular, if $r:=\gcd{\left(\frac{n}{2},12\right)}$ ($=\gcd{(m,12)}$), we have that:

\begin{enumerate}
\itemsep2pt \parskip0pt \parsep0pt

\item[(i)] if $r=1$ then $\mathcal{B}(n)=12$;
\item[(ii)] if $r=2$ then

\begin{tabular}{ll}
$  \bullet \,  \mathcal{B}(n)=6$  & if $n \not\equiv 28 \pmod{32}$, \\
 $\bullet \,  \mathcal{B}(n)=12$ & otherwise;
 \end{tabular}

\item[(iii)] if $r=3$ then 

\begin{tabular}{ll}
$\bullet \,  \mathcal{B}(n)= 4$ & if all  prime factors of $\frac{n}{6}$ congruent to $3\pmod 4$ \\ & occur  with even exponent, \\
$\bullet \,  \mathcal{B}(n)= 8$ & otherwise;
 \end{tabular}

\item[(iv)] if $r=4$ then 

\begin{tabular}{ll}
 $\bullet \, \mathcal{B}(n)= 3$ & if $\frac{n}{2}$ is a square, \\
 $\bullet \,  \mathcal{B}(n)= 6$  & if $\frac{n}{2}$ is not a square and $n\neq 4^k(8t+7)\, \forall k,t\in  \mathbb{Z}_{\geq0},$ \\
$\bullet \,  \mathcal{B}(n) = 9$ & otherwise;
 \end{tabular} 


\item[(v)] if $r=6$ then 

\begin{tabular}{ll}
$\bullet \,  \mathcal{B}(n)= 2$ & if $\frac{n}{12}$ is a square, \\
$\bullet \,  \mathcal{B}(n)=4$ & if $\frac{n}{12}$ is not a square and all prime factors of $\frac{n}{6}$ \\ & congruent to   $3\pmod 4$ occur with even exponent, \\
$\bullet \, \mathcal{B}(n)= 6$ & if $\frac{n}{12}$ is not a square,  at least one prime factor of $\frac{n}{6}$ \\ & congruent  to $3\pmod 4$ occurs with an odd exponent  \\  & and $n \not\equiv 28 \pmod{32}$, \\
$\bullet \,  \mathcal{B}(n)= 8$ & otherwise; 
 \end{tabular} 


\item[(vi)] if $r=12$ then 

\begin{tabular}{ll}
$\bullet \, \mathcal{B}(n)= 2$ & if $\frac{n}{12}$  is a square, \\
$\bullet \, \mathcal{B}(n)=3$ &  if $\frac{n}{2}$ is a square, \\
$\bullet \, \mathcal{B}(n)= 4$  & if none of the above holds and all prime factors  of $\frac{n}{6}$ \\ & congruent to $3\pmod 4$ occur with even exponent, \\
$\bullet \,\mathcal{B}(n)=6$  & if none of the above holds  and $n\neq  4^k(8t+7)\, \forall k,t\!\in \!\mathbb{Z}_{\geq0}$,\\
$\bullet \, \mathcal{B}(n)= 7$ & otherwise.
\end{tabular}

\end{enumerate}

\end{theoremL}
\begin{proof}
In $\mathcal{Z}_1$ we have 
$$
G_1= -mN_m + 2\sum_{k=1}^{m} (6k^2-m)N_{m-k} = 0,
$$
and so, in this set,
\begin{equation}\label{eq:B1}
N_m = 2 \sum_{k=1}^{m} \left(\frac{6k^2}{m}-1\right)N_{m-k} \in \mathbb{Z}_{\geq 0}.
\end{equation}
Hence, to find $\min_{\mathcal{Z}_1} F_1$, we start by substituting \eqref{eq:B1} in \eqref{f1},
obtaining
\begin{equation}\label{eq:B2}
F_1=\frac{12}{m}\sum_{k=1}^{m}k^2 N_{m-k}.
\end{equation}
Since $F_1$ is integer valued  on $\mathbb{Z}^{n+1}$, we need 
$$
\frac{12}{m}\sum_{k=1}^{m}k^2 N_{m-k} \in \mathbb{Z}.
$$
As $N_{0},\ldots, N_{m-1}\in \mathbb{Z}$, this is equivalent to requiring
$$
\sum_{k=1}^{m}k^2 N_{m-k} \equiv 0 \pmod{\frac{m}{r}},
$$
with $r:=\gcd{(m,12)}=\gcd{( \frac{n}{2} ,12)}\in \{1,2,3,4,6,12\}$. Note that this also implies that the expression on the right hand side of \eqref{eq:B1} is an integer and that
\begin{equation}\label{eq:pmod}
F_1\equiv 0 \pmod{\frac{12}{r}}.
\end{equation}

We then want to find the smallest  positive value of 
$$
\sum_{k=1}^{m}k^2 N_{m-k}
$$
which is a multiple of $\frac{m}{r}$ and such that
\begin{equation}\label{eq:B3}
 \sum_{k=1}^{m} \left(\frac{6k^2}{m}-1\right)N_{m-k}\geq 0,
\end{equation}
so that \eqref{eq:B1} is satisfied.  Then, by \eqref{eq:B2}, the  minimum  $\mathcal{B}(n)$ of $F_1$ on $\mathcal{Z}_1$ is obtained by multiplying this value by $\frac{12}{m}$.
\begin{remark}
Note that, when $m\leq 6$, condition \eqref{eq:B3} is always satisfied. Hence, the smallest multiple of $\frac{m}{r}$ that satisfies all the required conditions is $\frac{m}{r}$ itself (taking for instance $N_{m-1}=\frac{m}{r}$, $N_m=\frac{2(6-m)}{r}$ and all other $N_i$s equal to $0$), leading to 
$$
\mathcal{B}(n) = \frac{m}{r}\cdot \frac{12}{m}=\frac{12}{r},\quad \text{whenever} \quad n=2m \quad \text{with}\quad m\leq 6.
$$
\end{remark}
In general, we see that \eqref{eq:B3} is equivalent to
$$
\sum_{k=1}^{m} k^2 N_{m-k} \geq \frac{m}{6}\sum_{k=1}^{m} N_{m-k},
$$
so our goal is to find the smallest positive multiple of $\frac{m}{r}$ which can be written as
$$
\sum_{k=1}^{m} k^2 N_{m-k}
$$
and is greater or equal to 
$$
\frac{m}{6}\sum_{k=1}^{m} N_{m-k}.
$$
In other words, for each $m$, we want to find the smallest value of $l\in \mathbb{Z}_{>0}$ such that
\begin{equation}\label{eq:B4}
l \cdot \frac{m}{r} = \sum_{k=1}^{m} k^2 N_{m-k} \geq \frac{m}{6}\sum_{k=1}^{m} N_{m-k}.
\end{equation}
Note that the first sum in \eqref{eq:B4} is a sum of squares, possibly with repetitions (whenever one of the $N_{m-k}$s is greater than $1$), and that the sum on the right hand side of \eqref{eq:B4} is precisely the number of squares used in this representation of $l \cdot \frac{m}{r}$ as a sum of squares. 
We then want to find the  smallest value of $l\in \mathbb{Z}_{>0}$ such that 
\begin{equation}\label{eq:B5}
\sum_{k=1}^{m} N_{m-k} \leq  \frac{6 l}{r},
\end{equation}
where $\sum_{k=1}^{m} N_{m-k}$ is the smallest number of squares that is needed to represent the positive integer $l \cdot \frac{m}{r}$ as a sum of squares. We can then use the results in Section~\ref{sec:prent}.

When $r=1$ condition \eqref{eq:B5} becomes 
\begin{equation}\label{eq:B6}
\sum_{k=1}^{m} N_{m-k} \leq  6 l.
\end{equation}
Since, by Theorem~\ref{thm:Lagrange}, we know that every positive integer can be written as a sum of $4$ or fewer squares,  \eqref{eq:B6} can be achieved with $l=1$, since $\frac{m}{r}=m$ can be written as a sum of $4$ or fewer squares and then
$$
\sum_{k=1}^{m} N_{m-k} \leq 4 \leq  6 l = 6.
$$
We conclude that, when $r=1$, we always have $\mathcal{B}(n)=  \frac{12}{m}\cdot \frac{m}{r}=12$.

When $r=2$, condition  \eqref{eq:B5} becomes 
\begin{equation}\label{eq:B7}
\sum_{k=1}^{m} N_{m-k} \leq  3 l.
\end{equation}
Hence, if $\frac{m}{r}=\frac{m}{2}$ can be written as a sum of $3$ or fewer squares,  \eqref{eq:B7} can be achieved with $l=1$. Otherwise we need $l=2$, since then, by Theorem~\ref{thm:Lagrange}, the number $\frac{2 m}{r}=m$ can be written as a sum of $4$ or fewer squares and then
$$
\sum_{k=1}^{m} N_{m-k} \leq 4 \leq  3 l = 6.
$$
Note that, since $r=2$, the number $\frac{m}{2}$ cannot be a multiple of $4$ and so the condition
$$
\frac{m}{2} \neq 4^k(8t+7)\quad \text{for all} \quad k,t\in \mathbb{Z}_{\geq0}
$$
in Theorem~\ref{thm:Legendre} is, in this situation, equivalent to 
$$
\frac{m}{2} \neq 8t+7 \quad \text{for all} \quad t\in \mathbb{Z}_{\geq0}
$$
which, in turn, is equivalent to $m\not\equiv 14 \pmod{16}$ (i.e.\ $n\not\equiv 28 \pmod{32}$). 
Hence, by Theorem~\ref{thm:Legendre}, we conclude that,  when $r=2$, we have
$\mathcal{B}(n)=  \frac{12}{m}\cdot \frac{m}{2}=6$ if  $n\not\equiv 28 \pmod{32}$ and $\mathcal{B}(n)=   \frac{12}{m}\cdot \frac{2 m}{2} = 12$ otherwise.

When $r=3$, condition  \eqref{eq:B5} becomes 
\begin{equation}\label{eq:B8}
\sum_{k=1}^{m} N_{m-k} \leq  2 l.
\end{equation}
Hence, if $\frac{m}{r}$ is a square or a  sum of $2$  squares,  \eqref{eq:B8} can be achieved with $l=1$. Otherwise we need $l=2$, since then, by Theorem~\ref{thm:Lagrange}, the number $\frac{2 m}{r}=\frac{2 m}{3}$ can be written as a sum of $4$ or fewer squares.
Hence, by Theorem~\ref{thm:Euler}, we conclude that,  when $r=3$, $\mathcal{B}(n)=  \frac{12}{m}\cdot \frac{m}{3}=4$ if all prime factors of $ \frac{m}{3}$ congruent to $3 \pmod 4$ occur with even exponent and $\mathcal{B}(n)=  \frac{12}{m}\cdot \frac{2 m}{3} = 8$ otherwise.

When $r=4$, condition  \eqref{eq:B5} becomes 
\begin{equation}\label{eq:B9}
\sum_{k=1}^{m} N_{m-k} \leq  \frac{3 l}{2}.
\end{equation}
Hence, if $\frac{m}{r}=\frac{m}{4}$ is a square (or, equivalently, if $m$ is a square),  \eqref{eq:B9} can be achieved with $l=1$. Otherwise, if $\frac{2m}{r}=\frac{m}{2}$ can be written as a sum of $3$ or fewer squares,  \eqref{eq:B9} can be achieved with $l=2$. 
Otherwise, we need $l=3$ since then,  by Theorem~\ref{thm:Legendre}, the number $\frac{3 m}{3}=\frac{3m}{4}$ can be written as a sum of $4$ or fewer squares.

Hence, by Theorem~\ref{thm:Legendre}, we conclude that,  when $r=4$, we have $\mathcal{B}(n)= \frac{12}{m}\cdot \frac{m}{4}=3$ if  $m$ is a square, $\mathcal{B}(n)=\frac{12}{m}\cdot \frac{2 m}{4} = 6$ if $m$ is not a square and $\frac{m}{2}\neq  4^k(8t+7)$ for all $k,t\in \mathbb{Z}_{\geq0}$ (which, since $n$ is even,  is equivalent to $n \neq  4^k(8t+7)$ for all $k,t\in \mathbb{Z}_{\geq0}$), and $\mathcal{B}(n)= \frac{12}{m}\cdot \frac{3 m}{4} = 9$ in all other cases.

When $r=6$, condition  \eqref{eq:B5} becomes 
\begin{equation}\label{eq:B10}
\sum_{k=1}^{m} N_{m-k} \leq  l .
\end{equation}
Hence, if $\frac{m}{r}=\frac{m}{6}$ is a square, then \eqref{eq:B10} can be achieved with $l=1$. Otherwise, if $\frac{2m}{r}=\frac{m}{3}$ is a  square or a sum of $2$  squares,   \eqref{eq:B10} can be achieved with $l=2$.
If this is not the case and   $\frac{3m}{r}=\frac{m}{2}$ is a sum of $3$ or fewer squares, then  \eqref{eq:B10} can be achieved with  $l=3$. If this also does not hold then   \eqref{eq:B10} can only be achieved with $l=4$ since then, by Theorem~\ref{thm:Lagrange} the number $\frac{4m}{r}=\frac{2m}{3}$ can be written as a sum of $4$ or fewer squares.

Note that, since $r=6$, the number $\frac{m}{2}$ cannot be a multiple of $4$. Hence, condition 
$$
\frac{m}{2}\neq  4^k(8t+7) \quad \text{for all} \quad k,t\in \mathbb{Z}_{\geq0}
$$
in Theorem~\ref{thm:Legendre} is, in this situation, equivalent to 
$$
\frac{m}{2}\neq 8t+7 \quad \text{for all} \quad t\in \mathbb{Z}_{\geq0}
$$
which, in turn, is equivalent to $n\neq 28 \pmod{32}$.
Hence, by Theorems~\ref{thm:Legendre} and \ref{thm:Euler}, we conclude that,  when $r=6$, we have
$\mathcal{B}(n)= \frac{12}{m}\cdot \frac{m}{6}=2$ if   $ \frac{m}{6}=\frac{n}{12}$ is a square; otherwise $\mathcal{B}(n)=\frac{12}{m}\cdot \frac{2 m}{6} = 4$ if all prime factors of $\frac{m}{3}=\frac{n}{6}$ congruent to $3\pmod 4$ occur with even exponent; if none of these holds then $\mathcal{B}(n)= \frac{12}{m}\cdot \frac{3 m}{6} = 6 $ if $n\neq 28 \pmod{32}$ and $\mathcal{B}(n)= \frac{12}{m}\cdot \frac{4 m}{6} = 8$ otherwise.

When $r=12$, condition  \eqref{eq:B5} becomes 
\begin{equation}\label{eq:B11}
\sum_{k=1}^{m} N_{m-k} \leq  \frac{l}{2} .
\end{equation}
Hence, even if $\frac{m}{r}$ were a square, condition  \eqref{eq:B11} could never be achieved with $l=1$. If $\frac{2m}{r}=\frac{m}{6}=\frac{n}{12}$ is a  square,   \eqref{eq:B11} can be achieved with $l=2$.
If this is not the case and   $\frac{3m}{r}=\frac{m}{4}$ is a square (or, equivalently, if $m$ is a square), then  \eqref{eq:B11} can be achieved with  $l=3$. (Note that if $\frac{m}{4}$ is a square then $\frac{m}{6}$ is not a square.) If this also does not hold and $\frac{4m}{r}=\frac{m}{3}$ is a square or a sum of two squares, then  \eqref{eq:B11} can be achieved with $l=4$.
In none of the above holds and $\frac{5m}{r}=\frac{5m}{12}$ is a square or a sum of two squares then   \eqref{eq:B11} could be achieved with $l=5$. Note, however, that if $\frac{m}{3}$ is not a square nor a sum of two squares then, by Theorem~\ref{thm:Euler}, at least one prime factor of $\frac{m}{3}$ is congruent to $3 \pmod{4}$ and  occurs with  odd exponent. Then, since $5\neq 3 \pmod 4$, the number $\frac{5m}{12}$ also has this prime factor occurring  with the same odd exponent and so, in this situation, $\frac{5m}{12}$ cannot be written as a sum of $2$ or fewer squares, implying that this case is impossible. 

If none of the above conditions are true but  $\frac{6m}{r}=\frac{m}{2}=\frac{n}{4}$ is a sum of $3$ or fewer squares, then  \eqref{eq:B11} can be achieved with $l=6$.
If still $\frac{6m}{r}=\frac{m}{2}$ cannot be written as  a sum of $3$ or fewer squares then $\frac{7m}{r}=\frac{7m}{12}$ can, and so condition  \eqref{eq:B11} can be achieved with $l=7$. Indeed, if $\frac{m}{2}$ cannot be written as  a sum of $3$ or fewer squares, then
$$
\frac{m}{2} = 4^k(8t+7) \quad \text{for some} \quad k,t\in \mathbb{Z}_{\geq0},
$$ 
and $k\geq 1$ (since $m$ is  multiple of $4$);  then
$$
\frac{7m}{12} = \frac{14}{3} \cdot 4^{k-1} (8t+7),
$$
and so $8t+7 = 0\pmod 3$, implying that $t=1 \pmod 3$. Hence,
$$
\frac{7m}{12} =  \frac{14}{3} \cdot 4^{k-1} (24\,t^\prime+15) = 14 \cdot  4^{k-1} (8t^\prime+5) =  4^{k-1} (8t^{\prime\prime}+70) =  4^{k-1} (8t^{\prime\prime\prime}+6)
$$
for some $t^\prime, t^{\prime\prime}, t^{\prime\prime\prime}\in \Z_{\geq0}$ and so, by Theorem~\ref{thm:Legendre}, the number $\frac{7m}{12}$ can be represented by a sum of $3$ or fewer squares.

We conclude,  by Theorems~\ref{thm:Legendre} and \ref{thm:Euler}  that,  when $r=12$, we have
$\mathcal{B}(n)= \frac{12}{m}\cdot \frac{2 m}{12}=2$ if  $\frac{m}{6}=\frac{n}{12}$ is a square, $\mathcal{B}(n)= \frac{12}{m}\cdot \frac{3 m}{12} = 3$ if   $m=\frac{n}{2}$ is a square, and $\mathcal{B}(n)=\frac{12}{m}\cdot \frac{4 m}{12} = 4$ if neither $\frac{n}{12}$ nor $\frac{n}{2}$ are squares and all prime factors  of $\frac{n}{6}$  congruent to $3 \pmod 4$ occur with even exponent. If none of these conditions hold then $\mathcal{B}(n)= \frac{12}{m}\cdot \frac{6 m}{12} = 6$, if  $\frac{m}{2} \neq  4^k(8t+7)$ (or, equivalently, $n  \neq  4^k(8t+7)$) for any $k,t\in \mathbb{Z}_{\geq0}$, and $\mathcal{B}(n)= 7$ otherwise.

\end{proof}

\begin{remark}
In the Appendix we provide examples that show that all the cases listed in Theorem~\ref{thm:B} are possible.
\end{remark}

\section{Lower bound when $n$ is odd}
Here we compute the minimal value $\mathcal{B}(n)$ of the function $F_2$ restricted to $\mathcal{Z}_2$,  obtaining a lower bound for the number of fixed points of the $S^1$-action when $n$ is \emph{odd}.

\begin{theoremL}\label{thm:C}
Let $n=2m+1$ be an odd positive integer and let $\mathcal{B}(n)$ be the minimum of the function $F_2$ restricted to the set $\mathcal{Z}_2$, where $F_2$ and $\mathcal{Z}_2$ are respectively defined by \eqref{f2} and \eqref{z}. Then $\mathcal{B}(n)$ can take all values in the set $\{2,4,6,8,12,24\}$. In particular, if $r=\gcd{(\lfloor \frac{n}{2} \rfloor -1 , 12)}$ ($= \gcd{(m-1,12)}$), we have:
\begin{enumerate}
\itemsep2pt \parskip0pt \parsep0pt
\item[(i)] if $r\leq 4$ then $\mathcal{B}(n)=\frac{24}{r}$;
\item[(ii)] if $r=6$ then 

\begin{tabular}{ll}
$  \bullet \, \mathcal{B}(n)= 4$ & if  every prime factor of $\frac{n}{3}$ congruent to $3\pmod 4$ \\ & occurs with even exponent,  \\
 $\bullet \, \mathcal{B}(n)= 8$ & otherwise; 
\end{tabular}

\item[(iii)] if $r=12$ then 

\begin{tabular}{ll}
$  \bullet \,\mathcal{B}(n)= 2$ & if $\frac{n-3}{24}$  is a triangular number, \\
$  \bullet \, \mathcal{B}(n)=4$ & if  $\frac{n-3}{24}$ is not a triangular number and every prime \\ & factor of $\frac{n}{3}$   congruent to $3\pmod 4$ occurs with  \\ &  even exponent,  \\
$  \bullet \, \mathcal{B}(n)= 6$  & otherwise.
\end{tabular}
\end{enumerate}
\end{theoremL}

\begin{proof}
In $\mathcal{Z}_2$ we have 
$$
G_2= (1-m) N_m + \sum_{k=1}^{m} \Big(6k(k+1)-(m-1)\Big)N_{m-k} = 0.
$$

If $m=1$ then $G_2=12N_0=0$ implies that $N_0=0$ and so the minimum  of $F_2=2N_1$ on $\mathcal{Z}_2$ is $\mathcal{B}(3)=2$ (attained with $N_0=0$ and $N_2=1$). Note that here $\frac{n-3}{24}=0$ is a triangular number and we assume $r=\gcd{(m-1,12)}=\gcd{(0,12)}=12$.

If $m\neq 1$ then on  $\mathcal{Z}_2$ we have 
\begin{equation}\label{eq:C1}
N_m =  \sum_{k=1}^{m} \left(\frac{6k(k+1)}{m-1}-1\right)N_{m-k} \in \mathbb{Z}_{\geq 0}.
\end{equation}
Hence, to find $\min_{\mathcal{Z}_2} F_2$, we start by substituting \eqref{eq:C1} in \eqref{f2},
obtaining
\begin{equation}\label{eq:C2}
F_2=\frac{24}{m-1}\sum_{k=1}^{m}\frac{k(k+1)}{2} N_{m-k}.
\end{equation}
Since $F_2$ is integer valued and $N_m\in \mathbb{Z}$, we need 
$$
\frac{12}{m-1}\sum_{k=1}^{m}\frac{k(k+1)}{2} N_{m-k} \in \mathbb{Z}.
$$
Since $N_{0},\ldots, N_{m-1}\in \mathbb{Z}$, this is equivalently to requiring
$$
\sum_{k=1}^{m-1}\frac{k(k+1)}{2} N_{m-k}  \equiv 0 \pmod{\frac{m-1}{r}},
$$
with $r:=\gcd{(m-1,12)}=\gcd{( \lfloor \frac{n}{2}\rfloor -1,12)}\in \{1,2,3,4,6,12\}$.
Note that, by \eqref{eq:C2} this also implies that
\begin{equation}\label{eq:divodd}
F_2 \equiv 0 \pmod{\frac{24}{r}}.
\end{equation}

We then want to find the smallest  positive value of 
$$
\sum_{k=1}^{m}\frac{k(k+1)}{2} N_{m-k}
$$
which is a multiple of $\frac{m-1}{r}$ and such that
\begin{equation}\label{eq:C3}
 \sum_{k=1}^{m} \left(\frac{6k(k+1)}{m-1}-1\right)N_{m-k}\geq 0,
\end{equation}
so that \eqref{eq:C1} is satisfied. Then, by \eqref{eq:C2}, the  minimum  $\mathcal{B}(n)$ of $F_2$ on $\mathcal{Z}_2$ is obtained by multiplying this value by $\frac{24}{m-1}$.
\begin{remark}
Note that, when $m\leq 13$, condition \eqref{eq:C3} is always satisfied. Hence, the smallest multiple of $\frac{m-1}{r}$ that satisfies all the required conditions is $\frac{m-1}{r}$ itself, leading to 
$$
\mathcal{B}(n) = \frac{m-1}{r}\cdot \frac{24}{m-1}=\frac{24}{r},\quad \text{whenever} \quad n=2m+1 \quad \text{with}\quad m\leq 13.
$$
\end{remark}
In general, we see that \eqref{eq:C3} is equivalent to
$$
\sum_{k=1}^{m} \frac{k(k+1)}{2} N_{m-k} \geq \frac{m-1}{12}\sum_{k=1}^{m} N_{m-k},
$$
so our goal is to find the smallest positive multiple of $\frac{m-1}{r}$ which can be written as
$$
\sum_{k=1}^{m} \frac{k(k+1)}{2} N_{m-k}
$$
and is greater or equal to 
$$
\frac{m-1}{12}\sum_{k=1}^{m} N_{m-k}.
$$
In other words, for each $m$, we want to find the smallest value of $l\in \mathbb{Z}_{>0}$ such that
\begin{equation}\label{eq:C4}
l \cdot \frac{m-1}{r} = \sum_{k=1}^{m} \frac{k(k+1)}{2} N_{m-k} \geq \frac{m-1}{12}\sum_{k=1}^{m} N_{m-k}.
\end{equation}
Note that the first sum in \eqref{eq:B4} is a sum of triangular numbers, possibly with repetitions (whenever one of the $N_{m-k}$s is greater than $1$), and that the sum on the right hand side of \eqref{eq:C4} is precisely the number of triangular numbers used in this representation of $l \cdot \frac{m-1}{r}$ as a sum of triangular numbers. 
We then want to find the  smallest value of $l\in \mathbb{Z}_{>0}$ such that 
\begin{equation}\label{eq:C5}
\sum_{k=1}^{m} N_{m-k} \leq  \frac{12\, l}{r},
\end{equation}
where $\sum_{k=1}^{m} N_{m-k}$ is the smallest number of triangular numbers that is needed to represent the positive integer $l \cdot \frac{m-1}{r}$ as a sum of triangular numbers. We can therefore use the results in Section~\ref{sec:prent} concerning these numbers.

Since, by Theorem~\ref{thm:Gauss}, we know that every positive integer can be written as a sum of $3$ or fewer triangular numbers, condition \eqref{eq:C5} can be achieved with $l=1$ whenever $r\leq 4$ and then $\mathcal{B}(n)=  \frac{24}{m-1}\cdot \frac{m-1}{r}=\frac{24}{r}$.

When $r=6$, condition  \eqref{eq:C5} becomes 
\begin{equation}\label{eq:C6}
\sum_{k=1}^{m} N_{m-k} \leq  2 l.
\end{equation}
Hence, if $\frac{m-1}{r}=\frac{m-1}{6}$ can be written as a sum of $2$ or fewer triangular numbers,  \eqref{eq:C6} can be achieved with $l=1$. Otherwise we need $l=2$, since then, by Theorem~\ref{thm:Gauss}, the number $\frac{2 m}{r}=\frac{m}{3}$ can be written as a sum of $3$ or fewer squares and so
$$
\sum_{k=1}^{m} N_{m-k} \leq 3 \leq  2 l = 4.
$$
Hence, by Theorem~\ref{thm:Ewell}, we conclude that
$\mathcal{B}(n)=  \frac{24}{m-1}\cdot \frac{m-1}{6}=4$ if every prime factor of 
$$
4\left(\frac{m-1}{6}\right)+1 = \frac{2m+1}{3}=\frac{n}{3}
$$ 
congruent to $3\pmod 4$ occurs with even exponent  and $\mathcal{B}(n)=   \frac{24}{m-1}\cdot \frac{2 (m-1)}{6} = 8$ otherwise.

When $r=12$, condition  \eqref{eq:C5} becomes 
\begin{equation}\label{eq:C7}
\sum_{k=1}^{m} N_{m-k} \leq   l.
\end{equation}
Hence, if $\frac{m-1}{r}$ is a triangular number, then  \eqref{eq:C7} can be achieved with $l=1$.  Otherwise, if $\frac{2(m-1)}{r}=\frac{m-1}{6}$ can be written as a sum of $2$ or fewer triangular numbers,  \eqref{eq:C7} can be achieved with $l=2$. If this is not the case, we need $l=3$ since then,  by Theorem~\ref{thm:Gauss}, the number $\frac{3 (m-1)}{12}=\frac{m-1}{4}$ can be written as a sum of $3$ or fewer triangular numbers.

Hence, by Theorem~\ref{thm:Ewell}, we conclude that $\mathcal{B}(n)=\frac{24}{m-1}\cdot \frac{m-1}{12}=2$ if  $ \frac{m-1}{12}$ is a triangular number, $\mathcal{B}(n)=\frac{24}{m-1}\cdot \frac{2 (m-1)}{12} = 4$   if every prime factor of $\frac{2m+1}{3}$   congruent to $3\pmod 4$ occurs with even exponent  and $\mathcal{B}(n)= \frac{24}{m-1}\cdot \frac{3 (m-1)}{12} = 6$ in all other cases.
\end{proof}

\begin{remark}
In the Appendix we provide examples that show that all the cases listed in Theorem~\ref{thm:C} are possible.
\end{remark}

%
%
%
%

\section{Comparing with other bounds}
Although our lower bound $\mathcal{B}(n)$ does not, in general, increase with $n$, there are some values of $n$ for which $\mathcal{B}(n)$ is better than the lower bound $\lfloor \frac{n}{2}\rfloor+1$ proposed by Kosniowski \cite{Ko}. Indeed, the following is an easy consequence of Theorems~\ref{thm:B} and \ref{thm:C}.
\begin{prop}\label{prop:2}  Let $\mathcal{B}(n)$ be the lower bound for the number of fixed points of a  $J$-preserving circle action on a $2n$-dimensional compact connected almost complex manifold $(M,J)$ with $c_1c_{n-1}[M]=0$ obtained in Theorems~\ref{thm:B} and \ref{thm:C}. Then 
$\mathcal{B}(n)\geq \lfloor \frac{n}{2} \rfloor+1$ if and only if 
$$\dim M\in \{4,6,8,10,12,14,18,20,22,26,28,34,44,46,50,58,74,82\}.$$

In particular, Kosniowski's conjecture is true for these dimensions whenever $c_1c_{n-1}[M]=0$.
\end{prop}
There are also some values of $n$ for which $\mathcal{B}(n)$ is greater than $n$ and we recover the lower bound for K\"ahler (Hamiltonian) actions. 
\begin{prop}\label{prop:3}  Let $\mathcal{B}(n)$ be the lower bound for the number of fixed points of a  $J$-preserving circle action on a $2n$-dimensional compact connected almost complex manifold $(M,J)$ with $c_1c_{n-1}[M]=0$ obtained in Theorems~\ref{thm:B} and \ref{thm:C}. Then 
$\mathcal{B}(n)\geq n +1$ if and only if 
$$\dim M\in \{4,8,10,14,20,26,34\}.$$
\end{prop}

\section{Divisibility results for the number of fixed points}
In a letter to V. Gritsenko, Hirzebruch \cite{Hi99} obtains divisibility results for the Chern number $c_n[M]$ (the Euler characteristic of the manifold) under the assumption $c_1c_{n-1}[M]=0$ (or under the stronger assumption that  $c_1=0$ in integer cohomology). In particular he proves the following result.
\begin{theorem}[Hirzebruch]\label{thm:hirz}
Let $M$ be a $2n$-dimensional stably almost complex manifold. If $c_1c_{n-1}[M]=0$ then
\begin{itemize}
\item if $n \equiv 1$ or $5 \pmod 8$, the Chern number $c_n[M]$ is divisible by  $8$;
\item if $n \equiv 2, 6$ or $7  \pmod 8$, the Chern number $c_n[M]$ is divisible by $4$;
\item if $n \equiv  3$ or $4 \pmod 8$, the Chern number  $c_n[M]$ is divisible by  $2$.
\end{itemize}
\end{theorem}
If an almost complex manifold is equipped with an $S^1$-action preserving the almost complex structure  with a nonempty discrete fixed point set, 
we know that $c_n[M]$ is equal to the number of fixed points of the action (see for example \cite[Section 3]{GoSa12}). 
Therefore, we can also obtain divisibility results for $c_n[M]$ from the expressions of the functions $F_1$ and $F_2$ in \eqref{eq:B2} and \eqref{eq:C2} 
in the proofs of Theorems~\ref{thm:B} and \ref{thm:C}  (since $F_1$ and $F_2$ count the number of fixed points), 
and it is straightforward to see that Theorem \ref{thm:D} is a direct consequence of \eqref{eq:pmod} and \eqref{eq:divodd}.

In particular, we improve Hirzebruch's divisibility factors for  $c_n[M]$  whenever  $n\not\equiv 0 \pmod{3}$,  and we obtain the same factors otherwise.
\begin{theoremL}\label{thm:E}
Let $(M,J)$ be a $2n$-dimensional compact connected almost complex manifold equipped with a $J$-preserving  $S^1$\--action with 
nonempty, discrete fixed point set $M^{S^1}.$ If   $c_1c_{n-1}[M]=0$ and $n\not\equiv 0 \pmod{3}$ then
\begin{itemize}
\item  if $n \equiv 0 \pmod 8$,  then $\lvert M^{S^1}\rvert$ is divisible by  $3$;
\item if $n \equiv 1$ or $5 \pmod 8$, then $\lvert M^{S^1}\rvert$ is divisible by  $24$;
\item if $n \equiv 2, 6$ or $7  \pmod 8$, then  $\lvert M^{S^1}\rvert$ is divisible by $12$;
\item if $n \equiv 3$ or $4 \pmod 8$, then $\lvert M^{S^1}\rvert$ is divisible by  $6$.
\end{itemize}
\end{theoremL}
\begin{proof}
 If $n=2m$ is even, we can write 
 $$
 n\equiv 2k \pmod{8} \quad \text{with} \quad k\in \{0,1,2,3\}
 $$ and $m\equiv k \pmod{4}$. Moreover,  $n \not\equiv 0 \pmod{3}$ implies that $m \not\equiv 0 \pmod{3}$. Hence, if  $r=\gcd{(m,12)}$, we have 
 $$
 \begin{array}{cl} r= 4 & \text{if $k=0$},\\  r=1 & \text{if $m$ is odd (i.e.\ if $k=1$ or $3$),} \\  r=2 & \text{if $k=2$}.\end{array}
 $$ 
The result for even values of $n$ then  follows from Theorem~\ref{thm:D} since  the number of fixed points is divisible by $\frac{12}{r}$. Note that, if $n\equiv 0 \pmod{3}$, and consequently $m\equiv 0 \pmod 3$, then 
$$
\begin{array}{cl} r=12 &  \text{if $k=0$,} \\ r=3 &  \text{if $m$ is odd,} \\ r=6 & \text{if $k=2$},
\end{array}
$$
and we recover Hirzebruch's divisibility factors in Theorem~\ref{thm:hirz}.
 
  If $n=2m+1$ is odd, we can write 
 $$
 n\equiv 2k+1 \pmod{8} \quad \text{with} \quad k\in \{0,1,2,3\}
 $$ and $m-1\equiv k-1 \pmod{4}$. Moreover,  $n \not\equiv 0 \pmod{3}$  implies that $m -1 \not\equiv 0 \pmod{3}$. Hence, if  $r=\gcd{(m-1,12)}$, we have 
 $$
 \begin{array}{cl} r=4 & \text{if $k=1$},\\  r=1 & \text{if $m-1$ is odd (i.e.\ if $k=0$ or $2$),} \\  r=2 & \text{if $k=3$}.\end{array}
 $$ 
The result for odd values of $n$ then  follows from Theorem~\ref{thm:D} since  the number of fixed points is divisible by $\frac{24}{r}$. Note that, if $n\equiv 0 \pmod{3}$, and consequently $m-1\equiv 0 \pmod 3$, then 
$$
\begin{array}{cl} r=12 &  \text{if $k=1$,} \\ r=3 &  \text{if $m-1$ is odd,} \\ r=6 & \text{if $k=3$},
\end{array}
$$
and we recover Hirzebruch's divisibility factors in Theorem~\ref{thm:hirz}.
\end{proof}

In summary, we have the following table.
$$
\begin{tabular}{| c | c  | l |}
\hline 
$n \pmod{8}$ & \multicolumn{2}{| c |}{$\lvert M^{S^1}\rvert$ is divisible by}   \\
\hline
\multirow{2}{*}{$0$} & $1$ & if $n\equiv 0 \pmod{3}$ \\ \cline{2-3}
   & $3$ & otherwise \\
\hline
\multirow{2}{*}{$1$} & $8$ & if $n\equiv 0 \pmod{3}$ \\ \cline{2-3}
   & $24$ & otherwise \\ \hline
   \multirow{2}{*}{$2$} & $4$ & if $n\equiv 0 \pmod{3}$ \\ \cline{2-3}
   & $12$ & otherwise \\ \hline
      \multirow{2}{*}{$3$} & $2$ & if $n\equiv 0 \pmod{3}$ \\ \cline{2-3}
   & $6$ & otherwise \\ \hline  
   \multirow{2}{*}{$4$} & $2$ & if $n\equiv 0 \pmod{3}$ \\ \cline{2-3}
   & $6$ & otherwise \\ \hline
      \multirow{2}{*}{$5$} & $8$ & if $n\equiv 0 \pmod{3}$ \\ \cline{2-3}
   & $24$ & otherwise \\ \hline
   \multirow{2}{*}{$6$} & $4$ & if $n\equiv 0 \pmod{3}$ \\ \cline{2-3}
   & $12$ & otherwise \\ \hline
    \multirow{2}{*}{$7$} & $4$ & if $n\equiv 0 \pmod{3}$ \\ \cline{2-3}
   & $12$ & otherwise \\ \hline
\end{tabular}
$$

Under the stronger condition that $c_1=0$ in integer cohomology,  Hirzebruch was able to improve his divisibility factor for $c_n[M]$ when $n=2m$ with $m\equiv 1 \pmod{4}$ \cite{Hi99}.  
\begin{prop}[Hirzebruch]\label{prop:hirz}
If $M$ is a $2n$-dimensional stably almost complex manifold with $c_1=0$ and even $n=2m$ with $m\equiv 1 \pmod 4$, then $c_{n}[M]\equiv 0 \pmod 8$. 
\end{prop}

Knowing this, we are also able to improve our divisibility factor for $\lvert M^{S^1}\rvert$ under this condition.
\begin{theoremL}\label{thm:F}
Let $(M,J)$ be a $2n$-dimensional compact connected almost complex manifold equipped with a $J$-preserving  $S^1$\--action with 
nonempty, discrete fixed point set $M^{S^1}.$ If   $c_1=0$ with $n\equiv 2 \pmod 8$ and $n\not\equiv 0 \pmod{3}$, then
$\lvert M^{S^1}\rvert$ is divisible by $24$.
\end{theoremL}
\begin{proof}
By Theorem~\ref{thm:E} and Proposition~\ref{prop:hirz}, we have that $\lvert M^{S^1}\rvert \equiv 0 \pmod{12}$ and $\lvert M^{S^1}\rvert \equiv 0 \pmod{8}$ so the result follows.
\end{proof}
Using this result we can improve the lower bound for the number of fixed points given by $\mathcal{B}(n)$, provided the assumptions of  Theorem~\ref{thm:F} are satisfied.
\begin{theoremL}\label{thm:G}
Let $(M,J)$ be a $2n$-dimensional compact connected almost complex manifold equipped with a $J$-preserving  $S^1$\--action with 
nonempty, discrete fixed point set $M^{S^1}.$ If   $c_1=0$ with $n\equiv 2 \pmod 8$ and $n\not\equiv 0 \pmod{3}$, then
the number of fixed points is at least $24$.
\end{theoremL}
\begin{remark}
If $n=2m$ with $n\equiv 2 \pmod 8$ and $n\not\equiv 0 \pmod{3}$ then from Theorem~\ref{thm:B} we always have  $\mathcal{B}(n)=12$ since $\gcd{(m,12)}=1$ ($m$ is odd and is not a multiple of $3$).
\end{remark}

\section{Examples}\label{sec:ex}
We will now show that some of the lower bounds obtained in Theorems~\ref{thm:B} and \ref{thm:C} for the number of fixed points are sharp. For that, we will first state the following lemma which gives a way of producing infinitely many manifolds with $c_1c_{n-1}[M]=0$.

\begin{lemma}\label{lemma:salamon}
Let $M^{2m}$ and $N^{2n}$ be compact almost complex manifolds satisfying $c_1c_{m-1}[M]=c_1c_{n-1}[N]=0$. Then $c_1c_{m+n-1}[M\times N]=0$.
\end{lemma}
\begin{proof}
This follows from the fact that if, for any almost complex manifold $M^{2m}$ with $c_m[M]\neq0$, we set
$$
\gamma(M):=\frac{c_1c_{m-1}[M]}{c_m[M]},
$$
we have $\gamma(M\times N)=\gamma(M)+\gamma(N)$ (see \cite[Section 3]{S}).
\end{proof}

\begin{example}\label{ex:1}
There exists a $4$ dimensional almost complex manifold manifold $(N^4,J)$ with $c_1^2[N]=0$ that admits a  $J$-preserving circle action with $12$ fixed points (note that, since $n=2$, we have $\gcd{(\frac{n}{2},12)}=1$ and $\mathcal{B}(2)=12$). Indeed, from \eqref{eq:GS3} we can just take 
$$N^4=\mathbb{C}\mathbb{P}^2\#9\overline{\mathbb{C}\mathbb{P}}^2,$$
the $9$-point blow-up of $\mathbb{C}\mathbb{P}^2$ since  
$$b_2(N)=10\quad \text{and} \quad b_0(N)=b_4(N)=1,$$ 
so that, by \eqref{eq:GS3},  
$$c_1^2[N]=10\, b_0(N) - b_2(N)=0\quad \text{and}\quad  b_0(N)+b_2(N)+b_4(N)=12.$$ 
Taking the standard Hamiltonian circle action on $\mathbb{C}\mathbb{P}^2$ (with $3$ isolated fixed points) and blowing up successively at index $2$ fixed points, we can obtain a Hamiltonian circle action on $N$  with exactly $12$ fixed points.
\end{example}

\begin{example}\label{ex:2}
For $\dim M=6$ we can take $M=S^6$ with the almost complex structure induced by a vector product in $\R^7$ and equipped with the $S^1$-action induced by the action on $\R^7=\R\oplus\mathbb{C}^3$ given by
$$
\lambda \cdot (t,z_1,z_2,z_3)=(t,\lambda^n z_1, \lambda^m z_2,\lambda^{-(n+m)}z_3), \quad \lambda\in S^1, 
$$
with $t\in \R$, $z_1,z_2,z_3\in \mathbb{C}$, $m,n\in \Z\setminus\{0\}$ and $m+m\neq0$. This action has exactly $2$ fixed points and $N_1=N_2=1$ (note that $\mathcal{B}(3)=2$). 
\end{example}

\begin{example}\label{ex:3}
In any dimension, since  we can write every even positive integer $2n\geq 4$ as
$$2n = 2(2k+3l)=4k+6l,$$ 
for some $k,l\in \mathbb{Z}_{\geq 0}$,  we can take 
$$M=(N^4)^k\times (S^6)^l,$$
where $N^4$ is the $S^1$-manifold in Example~\ref{ex:1} and $S^6$ has the action in Example~\ref{ex:2}, to obtain an example of dimension $2n$. By Lemma~\ref{lemma:salamon} this almost complex manifold satisfies 
$$c_1c_{n-1}[M]=0,$$ 
and the diagonal circle action preserves the almost complex structure  and has $2^l\times 12^k$ fixed points.

If $k=l=1$ then $\dim M=10$ and the action has a minimal number of fixed points. Indeed, it has $24$ fixed points and $\mathcal{B}(5)=24$.

If $k=0$ and $l=2$ then $\dim M=12$ and the action has exactly $4$ fixed points so it also has a minimal number of fixed points (since $\mathcal{B}(6)=4$).

If $k=0$ and $l=3$ then $\dim M=18$ and the action has $8$ fixed points which is also   a minimal number ($\mathcal{B}(9)=8$).

\end{example}

\begin{remark}
It would be very interesting to find out if there exists an $8$ dimensional almost complex manifold $(M^8,J)$ with a $J$ preserving circle action with exactly $\mathcal{B}(4)=6$ fixed points. If this example could be constructed then $M^8\times S^6$ would give us a minimal  example  with $\mathcal{B}(7)=12$ fixed points.
\end{remark}

\begin{example}
Returning to Example~\ref{ex:3}, we see that, although the $S^1$-manifolds $(N^4)^k\times (S^6)^l$ do not always have a minimal number of fixed points, $\lvert M^{S^1}\rvert=2^l\times 12^k$, is always divisible by the factors predicted in Theorems~\ref{thm:hirz} and \ref{thm:E}.

Indeed, if $n$ is even and $k\neq 0$, then $\lvert M^{S^1}\rvert$ is a multiple of $12$ and so it is divisible by all the factors predicted in Theorems~\ref{thm:hirz} and \ref{thm:E}. If $n$ is even and $k=0$ then necessarily $n\equiv 0 \pmod{3}$.  Since $n=3l$ is even, we have $l>1$, and so $\lvert M^{S^1}\rvert$ is a multiple of $4$, again divisible by all the factors predicted in Theorem~\ref{thm:hirz}.

If $n$ is odd then necessarily $l>1$. If $k\neq 0$ then $\lvert M^{S^1}\rvert$ is a multiple of $24$ and so it is divisible by all the factors predicted in Theorems~\ref{thm:hirz} and \ref{thm:E}. If $k=0$ then necessarily $n\equiv 0 \pmod{3}$. We then have $n=3l$ and $2^l$ fixed points. If $l=1$ then $n\equiv 3 \pmod 8$ and  $\lvert M^{S^1}\rvert$ is divisible by $2$ (the factor predicted in Theorem~\ref{thm:hirz}); if $l=2$ then $n\equiv 6 \pmod 8$  and  $\lvert M^{S^1}\rvert$ is divisible by $4$, as predicted in Theorem~\ref{thm:hirz}; if $l\geq 3$ then  $\lvert M^{S^1}\rvert$ is a multiple of  $8$ and so it is divisible by all the factors given by Theorem~\ref{thm:hirz}.
%
\end{example}

\section{Final remarks} \label{fr}

The lower bound for the number of fixed points in Proposition~\ref{prop} follows from the fact that the momentum map $\mu$ 
is a Morse-Bott function, whose set of critical points 
$\operatorname{Crit}(\mu)$ is a submanifold of $M$, and coincides with the
fixed point set of the action. 
Thus, if it is not zero dimensional,  there are infinitely
many critical points of $\mu$ and the result
is obvious. If it is
zero dimensional, then $\mu$ is a perfect Morse function
(i.e.\ the Morse inequalities are equalities) because of the
following classical result: \emph{If 
$f$ is a Morse function on a compact and connected
manifold whose critical points  have only even 
indices, then it is a perfect Morse function} 
 \cite[Corollary 2.19 on page 52]{Nicolaescu2007}.

Let $N_k(\mu)$ be the number of critical points of $\mu$ of index $k$. The total number
of critical points of $\mu$ is
$$\sum_{k=0}^{2n}N_k(\mu) = 
\sum_{k=0}^{2n}{\rm b}_{k}(M),$$ where 
${\rm b}_k(M):=\dim\left({\rm H}^{k}(M, \mathbb{R}) \right)$ is the $k$th Betti number of $M$. 
The classes $[\omega^k]$ are nontrivial in
${\rm H}^{2k}(M, \mathbb{R})$ for $k=0, \ldots, n$,
so ${\rm b}_{2k}(M) \geq 1$,
and hence the number of 
critical points of $\mu$ is at least $n+1$.

One can try to use Theorem \ref{main} below to
deduce a result analogous to Proposition~\ref{prop} for circle valued 
momentum maps by replacing the Morse inequalities by the
Novikov inequalities (see \cite[Chapter 11, Proposition 2.4]{Pajitnov2006}, 
\cite[Theorem 2.4]{Farber2004}), if all the critical
points of $\mu$ are non-degenerate.

\begin{theorem}[McDuff, '88] \label{main}
Let the circle $S^1$ act 
symplectically on a compact connected symplectic 
manifold $(M, \sigma)$. 
Then either the action admits a standard 
momentum map or, if not, there
exists a $S^1$-invariant 
symplectic form $\omega$ on $M$   that 
admits a circle valued momentum map 
$\mu: M \rightarrow S^1$. 
Moreover, $\mu$ is a Morse-Bott-Novikov function 
and each connected component of 
$M^{S^1} = \operatorname{Crit}(\mu)$ has even index. 
If $ \sigma$  is
integral, then $ \omega= \sigma$. 
\end{theorem}

The number
of critical points of the circle-valued momentum
map $\mu$  in Theorem~\ref{main} is then $\sum_{k=0}^{2n} N_k(\mu)$. This number is at least
$$
\sum_{k=0}^{2n} \Big(\hat{{\rm b}}_k(M)+
\hat{{\rm q}}_k(M)+\hat{{\rm q}}_{k-1}(M) \Big),
$$
where  $\hat{{\rm b}}_k(M)$ is the rank of the 
$\mathbb{Z}((t))$-module 
${\rm H}_k(\widetilde{M},\mathbb{Z}) 
\otimes_{\mathbb{Z}[t, t^{-1}]}\mathbb{Z}((t))$, 
$\hat{{\rm q}}_k(M)$ is the torsion number of this 
module, and $\widetilde{M}$ is the pull back by 
$\mu: M \rightarrow \mathbb{R}/\mathbb{Z}$ of the 
principal $\mathbb{Z}$-bundle $t \in \mathbb{R} 
\mapsto [t] \in \mathbb{R}/\mathbb{Z}$. 
Unfortunately, this lower bound can be zero.
We refer to \cite[Sections 3 and 4]{PR12} for a detailed proof of
Theorem~\ref{main} and \cite[Remark 6]{PR12} for further details.
\smallskip
\medskip

\newpage
\appendix
\section{Tables}
\renewcommand{\arraystretch}{1.3}

\begin{table}[h!]
\begin{tabular}{| c | | c  | c  | c | l |}
\hline
{\color{blue!60!black} $n$}  & $m$ & $r=\gcd{(m,12)}$  & $\mathcal{B}(n)$ &  \\ \hline  \hline

{\color{blue!60!black} ${\bf 26}$} & $13$ & $1$ & $12$ &\\ \hline 

{\color{blue!60!black} ${\bf 20}$} & $10$ & $2$ & $6$ &  $20\not \equiv 28 \pmod{32}$\\\hline 

{\color{blue!60!black} ${\bf 28}$} & $14$ & $2$ & $12$ &  \\ \hline

{\color{blue!60!black} ${\bf 54}$} & $27$ & $3$ & $4$ &  $\frac{n}{6}= 3^2$ \\ \hline

{\color{blue!60!black} ${\bf 18}$} &$9$ & $3$ & $8$ & $\frac{n}{6}=3$ \\ \hline

{\color{blue!60!black} ${\bf 32}$} &$16$ & $4$ & $3$ & $\frac{n}{2}=4^2$  \\ \hline

\multirow{2}{*}{\color{blue!60!black} ${\bf 40}$} &\multirow{2}{*}{$20$} & \multirow{2}{*}{$4$} & \multirow{2}{*}{$6$} & $\frac{n}{2}$ is not a square and \\ & & & & $n=4\cdot 10\neq 4^k (8t+7), \forall k,t\in \Z_{\geq 0}$  \\ \hline

{\color{blue!60!black} ${\bf 112}$} & $56$ & $4$ & $9$ &  $\frac{n}{2}=56$ is not a square and $n=4^2\cdot 7$ \\ \hline

{\color{blue!60!black} ${\bf 108}$} & $54$ & $6$ & $2$ &  $\frac{n}{12}=3^2$  \\ \hline

{\color{blue!60!black} ${\bf 60}$} & $30$ & $6$ & $4$ &  $\frac{n}{12}=5$ is not a square and $\frac{n}{6}=2\cdot 5$ \\ \hline

\multirow{2}{*}{\color{blue!60!black} ${\bf 180}$} & \multirow{2}{*}{$90$}& \multirow{2}{*}{$6$} & \multirow{2}{*}{$6$} &  $\frac{n}{12}=15$ is not a square, $\frac{n}{6}=2\cdot 3 \cdot 5$ \\ & & & & and $n = 20+ 5 \cdot 32$ \\ \hline

\multirow{2}{*}{\color{blue!60!black} ${\bf 252}$} & \multirow{2}{*}{$126$}& \multirow{2}{*}{$6$} & \multirow{2}{*}{$8$} &  $\frac{n}{12}=21$ is not a square, $\frac{n}{6}=2\cdot 3 \cdot 7$ \\ & & & & and $n = 28+ 7 \cdot 32$ \\ \hline

{\color{blue!60!black} ${\bf 48}$} & $24$ & $12$ & $2$ &  $\frac{n}{12}=2^2$  \\ \hline

{\color{blue!60!black} ${\bf 72}$} & $36$ & $12$ & $3$ &  $\frac{n}{12}=6$ is not a square and $\frac{n}{2}=6^2$  \\ \hline

{\color{blue!60!black} ${\bf 24}$} & $12$ & $12$ & $4$ &  $\frac{n}{12}$, $\frac{n}{2}$ are not  squares and $\frac{n}{6}=2^2$  \\ \hline

\multirow{2}{*}{\color{blue!60!black} ${\bf 144}$} & \multirow{2}{*}{$72$}& \multirow{2}{*}{$12$} & \multirow{2}{*}{$6$} &  $\frac{n}{12}=12$ and $\frac{n}{2}=72$ are not  squares,  \\ & & & & $\frac{n}{6}=2^3\cdot 3$ and  $n = 4^2(8+1)$ \\ \hline

\multirow{2}{*}{\color{blue!60!black} ${\bf 1008}$} & \multirow{2}{*}{$504$}& \multirow{2}{*}{$12$} & \multirow{2}{*}{$7$} &  $\frac{n}{12}=84$, $\frac{n}{2}=504$ are not  squares,  \\ & & & & $\frac{n}{6}=2^3\cdot 3\cdot 7$ and  $n = 4^2(8\cdot 7+ 7)$ \\ \hline
\end{tabular}
\vspace{.3cm}

\caption{Examples that illustrate all possible lower bounds of $\lvert M^{S^1}\rvert$ listed in Theorem~\ref{thm:B} when $n=\frac{1}{2}\dim M$ is even (by increasing order of $r$)}
\end{table}


\newpage
\renewcommand{\arraystretch}{1.3}

\begin{table}[h!]
\begin{tabular}{| c | | c  | c  | c | l |}
\hline
{\color{blue!60!black} $n$}  & $m$ & $r=\gcd{(m-1,12)}$  & $\mathcal{B}(n)$ & \\ \hline  \hline

{\color{blue!60!black} ${\bf 39}$} & $19$ & $6$ & $4$ & $\frac{n}{3}=13$ \\ \hline 

{\color{blue!60!black} ${\bf 63}$} & $31$ & $6$ & $8$ &  $\frac{n}{3}=3\cdot 7$ \\ \hline 

{\color{blue!60!black} ${\bf 75}$} & $37$ & $12$ & $2$ & $\frac{n-3}{24}=\frac{2\cdot 3}{2}$\\ \hline

\multirow{2}{*}{{\color{blue!60!black} ${\bf 51}$}}& \multirow{2}{*}{$25$} & \multirow{2}{*}{$12$} & \multirow{2}{*}{$4$} &  $\frac{n-3}{24}= 2$ is not triangular \\ & & & & and $\frac{n}{3}=17$ \\ \hline

\multirow{2}{*}{{\color{blue!60!black} ${\bf 99}$}}& \multirow{2}{*}{$49$} & \multirow{2}{*}{$12$} & \multirow{2}{*}{$6$} &  $\frac{n-3}{24}= 4$ is not triangular \\ & & & & and $\frac{n}{3}=3\cdot 11$ \\ \hline

\end{tabular}
\vspace{.3cm}

\caption{Examples that illustrate the possible lower bounds of $\lvert M^{S^1}\rvert$  when $n=\frac{1}{2}\dim M$ is odd,  for nontrivial cases listed in theorem~\ref{thm:C} ($r=6$  or $12$) }
\end{table}
\renewcommand{\arraystretch}{1}

\renewcommand{\arraystretch}{1}

\medskip\noindent

\begin{thebibliography}{9999}


\bibitem[CKP12]{CKP12} Cho, H. W., Kim, J. H., and Park, H. C., On the conjecture of Kosniowski, \emph{Asian J. Math.} {\bf 16} (2012), 271--278.

\bibitem[D]{D}  L. E. Dickson, \emph{History of the Theory of Numbers}, Vol. 2, Chelsea, New York, 1952.

\bibitem[E92]{E} J. A. Ewell,  On sums of triangular numbers and sums of squares,  \emph{Amer. Math. Monthly},  {\bf 99} (1992) 752--757.

\bibitem[Fa04]{Farber2004} Farber, M., \textit{Topology of Closed One-Forms}, Mathematical Surveys and Monographs, \textbf{108}, American Mathematical Society, 2004.

\bibitem[Fe01]{Feldman2001} Feldman, K.E., Hirzebruch genera of manifolds supporting a Hamiltonian circle action (Russian),  \textit{Uspekhi Mat. Nauk}  \textbf{56}(5) (2001), 187--188; translation in \textit{Russian Math. Surveys}  \textbf{56} (5) (2001), 978\--979.

\bibitem[Fr59]{Frankel1959} Frankel, T., Fixed points and torsion on K\"ahler manifolds, \textit{Ann. Math.} \textbf{70}(1) (1959), 1--8.

\bibitem[FP09]{FP09}  Fine, J. and  Panov, D.,  Symplectic Calabi-Yau manifolds, minimal surfaces and the hyperbolic geometry of the conifold, \emph{J. Differential Geom.} {\bf 82} (2009), 155--205. 

\bibitem[GoSa12]{GoSa12} Godinho L. and Sabatini S., New tools for classifying Hamiltonian circle actions with isolated fixed points, to appear in
\emph{Foundations of Computational Mathematics}. \\ DOI: 10.1007/s10208-014-9204-1, 
\href{http://arxiv.org/abs/1206.3195}{{\tt arxiv:1206.3195 [math.SG]}}.

\bibitem[Ha84]{Ha84} Hattori, A., $S^1$\--actions on unitary manifolds and quasi-ample line bundles, \emph{J. Fac. Sci. Univ. Tokyo sect. IA, Math.} {\bf 31} (1984) 433\--486.

\bibitem[Hi99]{Hi99} Hirzebruch, F., On the Euler characteristic of manifolds with $c_1=0$. A letter to V. Gritsenko, \emph{Algebra i Analiz} {\bf 11} (1999), no. 5, 126\--129; translation in \emph{St. Petersburg Math. J.} {\bf 11} (2000), no. 5, 805\--807. 

\bibitem[Ko79]{Ko}  Kosniowski, C., Some formulae and conjectures associated to circle actions, Topology Symposium, Siegen 1979 (Prof. Symps., Univ. Siegen, 1979(, pp 331\--339. Lecture Notes in Math 788, Springer, Berlin 1980. 

\bibitem[LL10]{LiLi2010}  Li, P. and Liu, K., Some remarks on circle action on manifolds \emph{Math. Res. Letters}. {\bf 18} (2011) 435\--446.

\bibitem[Ni07]{Nicolaescu2007} Nicolaescu, L., \textit{An Invitation to Morse Theory}, Universitext, Springer-Verlag, New York, 2007.

\bibitem[Pa06]{Pajitnov2006} Pajitnov, A., \textit{Circle-valued Morse Theory}, de Guyter Studies in Mathematics, \textbf{32}, Walter de Gruyter, Berlin, New York, 2006.

\bibitem[PR12]{PR12} Pelayo, \'A., Ratiu, T. S., Circle\--valued momentum maps for symplectic periodic flows. \emph{Enseign. Math. (2)} {\bf 58} (2012) 205\--219.

\bibitem[PT11]{PeTo2010}Pelayo, A. and Tolman, S., Fixed points of symplectic periodic flows, \emph{Ergod. Theory and Dyn. Syst.} {\bf 31}  (2011) 1237\--1247.

\bibitem[S]{S} Salamon, S. M., \emph{Cohomology of K\"{a}hler manifolds with $c1=0$}. Manifolds and geometry (Pisa, 1993), 294-310, Sympos. Math., XXXVI, Cambridge Univ. Press, Cambridge, (1996).

\end{thebibliography}
 \end{document}